\documentclass[a4paper]{amsart}
\usepackage[utf8]{inputenc}
\usepackage[T1]{fontenc}
\usepackage{graphicx}
\usepackage{authblk}
\usepackage{amsaddr}
\usepackage{tikz} 
\usetikzlibrary{calc} 

\usepackage{amsmath,amssymb}
\usepackage{amsthm}
\usepackage{thmtools}
\usepackage{hyperref}

\newcommand{\Nat}{\mathbb{N}}
\newcommand{\Real}{\mathbb{R}}
\newcommand{\dd}{\mathrm{d}}
\newcommand{\Ind}{\mathbb{I}}
\newcommand{\on}{|_}

\newcommand{\func}[3] {#1~:~#2~\rightarrow~#3}

\newcommand{\Prob}{\mathbb{P}}
\newcommand{\Exp}{\mathbb{E}}
\newcommand{\Var}{\mathbb{V}}
\newcommand{\eqD}{\overset{\rm d}{\sim}}
\newcommand{\leD}{\leq_d}

\newcommand{\Bin}{\text{\rm Bin}}
\newcommand{\Mult}{\text{\rm Mult}}

\declaretheorem[name=Definition,numberwithin=section]{definition}
\declaretheorem[name=Remark,sibling=definition]{remark}

\declaretheorem[name=Proposition,sibling=definition]{proposition}
\declaretheorem[name=Lemma,sibling=definition]{lemma}
\declaretheorem[name=Theorem,sibling=definition]{theorem}
\declaretheorem[name=Corollary,sibling=definition]{corollary}

\newcommand{\dist}[3][] {d_{#1}(#2,#3)}
\newcommand{\diam}[1] {\text{\rm diam}(#1)}
\newcommand{\nhood}[1][] {
\def\ARGI{{#1}}%
\nhoodI}
\newcommand{\nhoodI}[2][] {\mathcal{N}^{\ARGI}_{#1}(#2)}
\newcommand{\Nhood}[1][] {
\def\ARGI{{#1}}%
\NhoodI}
\newcommand{\NhoodI}[2][] {\mathcal{N}^{\ARGI}_{#1}[#2]}

\newcommand{\df}[1] {\phi_{#1}} 
\newcommand{\dsed} {\Delta} 

\title{Local symmetry in random graphs}

\author[J. E. Simões]{Jefferson Elbert Simões}
\author[D. R. Figueiredo]{Daniel R. Figueiredo}
\author[V. C. Barbosa]{Valmir C. Barbosa}
\address{Systems Engineering and Computer Science Program, COPPE \\
Federal University of Rio de Janeiro, Rio de Janeiro, Brazil}
\email{\{elbert,daniel\}@land.ufrj.br, valmir@cos.ufrj.br}
\thanks{This work has been partially funded
by CAPES, CNPq, and FAPERJ BBP grants.}
\begin{document}

\begin{abstract}
Quite often real-world networks can be thought of as being symmetric,
in the abstract sense that vertices can be found
to have similar or equivalent structural roles.
However, traditional measures of symmetry in graphs
are based on their automorphism groups,
which do not account for the similarity of local structures.
We introduce the concept of \emph{local symmetry},
which reflects the structural equivalence of the vertices' egonets.
We study the emergence of asymmetry in the Erdős-Rényi random graph model
and identify regimes of both
asymptotic local symmetry and asymptotic local asymmetry.
We find that local symmetry persists at least
to an average degree of $n^{1/3}$
and local asymmetry emerges at an average degree not greater than $n^{1/2}$,
which are regimes of much larger average degree
than for traditional, global asymmetry.
\end{abstract}

\maketitle

\section{Introduction}

Graphs have become some of the most versatile mathematical objects,
capable of representing a wide range of real-world structures
such as the Internet, neural networks and
scientific collaboration networks, to name a few.
A natural question to pose in many of these scenarios is:
what is the meaning of symmetry in the context of graphs?
Several definitions have been proposed for this term,
each more adequate to a certain context or application.
Usually, such definitions are based upon transformations over graphs
that preserve certain properties of their structures,
the most traditional one being the existence of non-trivial automorphisms.
Such automorphisms allow us to classify the vertices of a graph
according to their ``role'' in its structure.

The concept of symmetry is closely tied to
that of \emph{structural identity}~\cite{StructIdentEquivIndivSocNet},
which is the identification of vertices based on
features of the network structure and their relationships to other vertices.
Structural identity can be applied in several contexts,
such as network privacy~\cite{Anonymized-Hidden-Steganography}.
For instance, data on social networks including personal information
is usually anonymized by the removal of labels,
in an attempt to preserve the privacy of their members (vertices).
If such network is symmetric, vertices with equal structural roles
cannot be distinguished without some kind of side information,
which makes them more likely to resist attempts at
deanonymization~\cite{ResistingStructReident}.

However, in several applied contexts,
the role of a vertex in the network structure,
and as a consequence its structural identity,
is intuitively related to its vicinity in the network,
rather than the whole network.
For instance, in neural networks (in which vertices correspond to neurons,
and edges to the synapses between them),
different areas of the brain become responsible for
specific functions, such as memory or motor coordination~\cite{NetworksBrain}.
In another example, in networks such as the web,
nodes can be assigned roles such as \emph{hub} or \emph{authority},
depending on how they are connected to their own
vicinities~\cite{AuthSourcesHyperlinkEnvironment}.

Our contributions are as follows.
We propose a definition of \emph{local symmetry},
based on the structural similarity of neighborhoods around each vertex.
Our definition naturally induces a hierarchy of symmetries,
which progressively use more information for classifying vertices,
ultimately culminating in the traditional, automorphism-based symmetry,
which we call \emph{global symmetry} in the context of this work.
Furthermore, we study the emergence of our base form of symmetry
in the context of the Erdős-Rényi random graph.
We find that, relative to global symmetry, asymptotic local symmetry occurs
for graphs with much larger average degrees,
in particular for degrees growing as fast as $n^{1/3}$ (\autoref{1-loc-sym}),
while asymptotic local asymmetry eventually emerges,
at most at degrees slightly larger than $n^{1/2}$ (\autoref{asym}).

\section{Local symmetry}
The abstract concept of symmetry
is traditionally embodied in the context of graphs
by the notions of \emph{isomorphism} and \emph{automorphism}.
Two graphs $G_1=(V_1,E_1)$ and $G_2=(V_2,E_2)$ are said to be
\emph{isomorphic} if there is a function $\func{f}{V_1}{V_2}$,
called an \emph{isomorphism} between $G_1$ and $G_2$,
that is bijective and precisely maps the edges of $G_1$ into edges of $G_2$.
This is a known equivalence relation on the set of all graphs,
and it can be thought of as identifying graphs with their edge structure,
ignoring the nature of their vertices or the interpretation of their edges.
Indeed, this view is implicit in traditional graph theory expressions
such as ``up to isomorphism''.

An isomorphism between a graph $G=(V,E)$ and $G$ itself is called
an \emph{automorphism} of $G$.
Every graph has a so-called \emph{trivial} automorphism given by
the identity function $\mathbb{I}_V$ over $V$, 
and a graph having no other automorphisms is said to be
\emph{asymmetric}~\cite{AsymGraphs}.
The set of automorphisms of a graph possesses a group structure
when coupled with the operation of composition,
which induces equivalence classes in its vertex set.
Again, there is a natural interpretion to this fact:
vertices can be grouped according to their placement in the graph structure,
such that vertices in the same class are ``structurally indistinguishable'',
at least without additional information about the identities of
(potentially all) the remaining vertices.
This interpretation leads us to the following definitions:

\begin{definition}
Given a graph $G=(V,E)$,
two vertices $v_1,v_2\in V$ are \emph{globally symmetric}
if there is an automorphism $f$ of $G$ such that $f(v_1) = v_2$.
\label{simet-gl-vertices}
\end{definition}

\begin{definition}
Let $G=(V,E)$ be a graph. Then $G$ is said to be \emph{globally symmetric}
if there are $u,v\in V$ distinct and globally symmetric.
\label{simet-gl-1grafo}
\end{definition}

\begin{remark}
$G$ is globally symmetric if and only if
it has at least one non-trivial automorphism.
\label{trad-simet}
\end{remark}

Throughout this paper, we will use the term ``globally asymmetric''
both for pairs of vertices and for graphs that are not globally symmetric.
Alternative definitions, specifically in the literature of probability theory,
have employed the simpler term \emph{symmetric graph}
for graphs satisfying \autoref{trad-simet}~\cite{AsymGraphs},
but this conflicts with at least two known graph-theoretic definitions
for the term
``symmetric graph''~\cite[p.104]{Biggs-AlgebGraphTheory}~\cite{PetersenGraph}.
%
Thus, introducing a new term in this definition may help avoid ambiguities.

One of the greatest difficulties in studying symmetry in graphs
lies in quantifying what it means for a graph to be ``almost symmetric''.
The usefulness of such quantification comes from the fact that,
in many real-world networks, one can identify intuitively equivalent vertices that the definition of global symmetry falls short in capturing.
For instance, consider the network in \autoref{fig:quase}.
Intuitively, the vertices $u$ and $v$ can be seen as small ``hubs'',
who share almost equally the role of connecting the remaining,
``peripheral'' vertices, which are also intuitively equivalent to each other.
However, since $u$ and $v$ have different degrees,
no automorphism is able to map them onto one another,
or map a peripheral neighbor of $u$ onto a peripheral neighbor of $v$.
Thus, the equivalence classes of vertices found on this graph
do not capture this intuitive notion of symmetry.

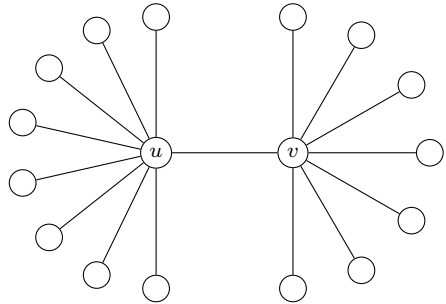
\begin{figure}[h]
\centering
        \begin{tikzpicture}[scale=0.9,every node/.style={circle,minimum size=1em,inner sep=2}]
        \node[draw=black] at (0,0) (u) {\footnotesize $u$};
        \node[draw=black] at (2,0) (v) {\footnotesize $v$};
        \path[draw=black] (u) -- (v);
        \def\uneigh{8}
        \foreach \un in {1,...,\uneigh}{
            \node[draw=black] at ($ (u) + ({90+(\un-1)*180/(\uneigh-1)}:2) $) (u\un) {};
            \path[draw=black] (u) -- (u\un);
        }
        \def\vneigh{7}
        \foreach \vn in {1,...,\vneigh}{
            \node[draw=black] at ($ (v) + ({90-(\vn-1)*180/(\vneigh-1)}:2) $) (v\vn) {};
            \path[draw=black] (v) -- (v\vn);
        }
        \end{tikzpicture}

\caption[Example of ``almost symmetry''.]{ Example of ``almost symmetric'' graph. Vertices $u$ and $v$ connect their neighbors to the remainder of the graph.}
\label{fig:quase}
\end{figure}

While several measures have been proposed
as a more flexible notion of symmetry~\cite{AsymGraphs,MaxSymSubgraphs},
no proposal seems to have achieved wide acceptance,
because of both the computational complexity of calculating such measures
and the absence of knowledge about their relationship
to real-world networks and applications.

A second difficulty, which we address in this work
and that has barely been analyzed in the literature,
lies in identifying vertices that (again intuitively)
have equivalent local structures but are globally distinguishable.
The graph in \autoref{fig:loc-eq} illustrates this idea well.
In this figure, we highlight two isomorphic induced subgraphs.
However, since both these subgraphs are part of a larger graph,
the existence of an automorphism that reflects such symmetry
depends on the edge pattern of the remainder of the graph,
and in this particular example the desired automorphism does not exist.
In other words, even though these subgraphs are intuitively symmetric,
since they are connected differently to the graph,
they are not globally symmetric.

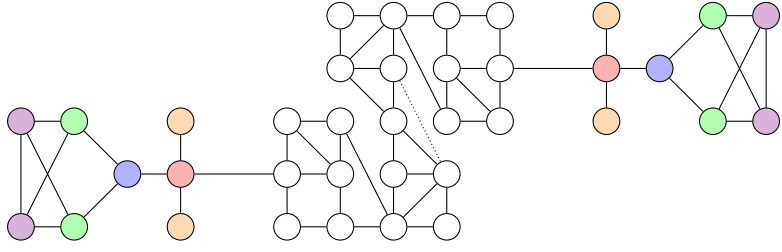
\begin{figure}[ht]
\centering
        \begin{tikzpicture}[scale=0.7,every node/.style={circle,minimum size=1em,inner sep=2}]
        \node[draw=black, fill=blue!30] at (0,-1) (a1) {};
        \node[draw=black, fill=green!30] at ($ (a1) + (-1,-1) $) (b1) {};
        \node[draw=black, fill=green!30] at ($ (a1) + (-1,1) $) (c1) {};
        \node[draw=black, fill=violet!30] at ($ (b1) + (-1,0) $) (d1) {};
        \node[draw=black, fill=violet!30] at ($ (c1) + (-1,0) $) (e1) {};
        \node[draw=black, fill=red!30] at ($ (a1) + (1,0) $) (f1) {};
        \node[draw=black, fill=orange!30] at ($ (f1) + (0,1) $) (g1) {};
        \node[draw=black, fill=orange!30] at ($ (f1) + (0,-1) $) (h1) {};
        \path[draw=black] (e1) -- (b1) -- (a1) -- (c1) -- (d1) (b1) -- (d1) -- (e1) -- (c1) (a1) -- (f1) -- (g1) (h1) -- (f1);

        \node[draw=black, fill=blue!30] at (10,1) (a2) {};
        \node[draw=black, fill=green!30] at ($ (a2) + (1,-1) $) (b2) {};
        \node[draw=black, fill=green!30] at ($ (a2) + (1,1) $) (c2) {};
        \node[draw=black, fill=violet!30] at ($ (b2) + (1,0) $) (d2) {};
        \node[draw=black, fill=violet!30] at ($ (c2) + (1,0) $) (e2) {};
        \node[draw=black, fill=red!30] at ($ (a2) + (-1,0) $) (f2) {};
        \node[draw=black, fill=orange!30] at ($ (f2) + (0,1) $) (g2) {};
        \node[draw=black, fill=orange!30] at ($ (f2) + (0,-1) $) (h2) {};
        \path[draw=black] (e2) -- (b2) -- (a2) -- (c2) -- (d2) (b2) -- (d2) -- (e2) -- (c2) (a2) -- (f2) -- (g2) (h2) -- (f2);

        \node[draw=black] at (3,0) (a3) {};
        \node[draw=black] at ($ (a3) + (1,0) $) (b3) {};
        \node[draw=black] at ($ (b3) + (1,0) $) (c3) {};
        \node[draw=black] at ($ (c3) + (1,0) $) (d3) {};
        \node[draw=black] at ($ (d3) + (1,0) $) (e3) {};
        \node[draw=black] at ($ (b3) + (0,1) $) (f3) {};
        \node[draw=black] at ($ (c3) + (0,1) $) (g3) {};
        \node[draw=black] at ($ (d3) + (0,1) $) (h3) {};
        \node[draw=black] at ($ (e3) + (0,1) $) (i3) {};
        \node[draw=black] at ($ (a3) + (0,-1) $) (j3) {};
        \node[draw=black] at ($ (b3) + (0,-1) $) (k3) {};
        \node[draw=black] at ($ (c3) + (0,-1) $) (l3) {};
        \node[draw=black] at ($ (d3) + (0,-1) $) (m3) {};
        \node[draw=black] at ($ (f3) + (0,1) $) (n3) {};
        \node[draw=black] at ($ (g3) + (0,1) $) (o3) {};
        \node[draw=black] at ($ (h3) + (0,1) $) (p3) {};
        \node[draw=black] at ($ (i3) + (0,1) $) (q3) {};
        \node[draw=black] at ($ (j3) + (0,-1) $) (r3) {};
        \node[draw=black] at ($ (k3) + (0,-1) $) (s3) {};
        \node[draw=black] at ($ (l3) + (0,-1) $) (t3) {};
        \node[draw=black] at ($ (m3) + (0,-1) $) (u3) {};
        
        \path[draw=black] (f1) -- (j3) -- (k3) -- (b3) -- (a3) -- (j3) -- (r3) -- (s3) -- (t3) -- (u3) -- (m3) -- (l3) -- (c3) (f2) -- (i3) -- (h3) -- (d3) -- (e3) -- (i3) -- (q3) -- (p3) -- (o3) -- (n3) -- (f3) -- (g3) -- (c3)
        (a3) -- (k3) -- (s3) (b3) -- (t3) -- (m3) -- (c3) (t3) -- (l3)
        (e3) -- (h3) -- (p3) (d3) -- (o3) -- (f3) -- (c3) (o3) -- (g3);
        \path[draw=black, densely dotted] (g3) -- (m3);
        \end{tikzpicture}

\caption[Example of local equivalence.]{ Example of equivalence between local structures in a graph.}\label{fig:loc-eq}
\end{figure}

Note that, regardless of the structure of the remainder of the graph,
and despite the fact that the intuitive equivalence
between the two local structures continues to exist,
the two aforementioned difficulties can happen simultaneously.
For instance, in the example of \autoref{fig:loc-eq},
by removing the single edge identified by a dashed line and
located at least at distance 5 from the highlighted subgraphs,
the graph acquires an automorphism that maps these subgraphs precisely.
Hence, the existence of this automorphism ---
a global mapping satisfying local restrictions ---
may be sensitive to changes that seem unpretentious
and unrelated to the local structures of interest.

To proceed more precisely, we must first delineate
which regions we are interested in analyzing.
Given a graph $G=(V,E)$ and a vertex set $S\subseteq V$,
we define $\nhood[][G]{S}$, the \emph{open neighborhood} of $S$ in $G$:
$$\nhood[][G]{S} = \{v\in V~:~\dist{v}{S}\leq 1\},$$
and the \emph{closed neighborhood} of $S$ in $G$:
$$\Nhood[][G]{S} = G[\nhood[][G]{S}],$$
where, for every set $A\subseteq V$ of vertices,
$\dist{v}{A}$ is the smallest distance between $v$ and some vertex of $A$,
and $G[A]$ is the subgraph induced by the vertices of $A$. 
Naturally, $\nhood[][G]{\{v\}}$ (or, by simplicity, $\nhood[][G]{v}$)
is the set that comprises $v$ and its neighbors,
and $\Nhood[][G]{\{v\}}$ (or $\Nhood[][G]{v}$)
is the subgraph induced by $v$ and its neighbors.
We will omit the index $_G$ when the graph at hand is clear.

This definition of neighborhood is traditional
in the field of graph theory\footnote{Actually,
the convention in graph theory is to define
$\nhood[][G]{S} = \{v\in V~:~\dist{v}{S}= 1\}$.
The reason we choose this alternate definition should be clear
in the remainder of this text.},
and we can extend it to include not only vertices at distance 1,
but at distance $k$.
We achieve this goal easily by defining
the \emph{open $k$-neighborhood} of $S$ in $G$:
$$\nhood[k][G]{S} = \{v\in V~:~\dist{v}{S}\leq k\}.$$
Note that $\nhood[0]{S} = S$, $\nhood[1]{S} = \nhood{S}$ and, recursively,
$\nhood[k]{S} = \nhood[]{\nhood[k-1]{S}}$.
We define the \emph{closed $k$-neighborhood} of $S$ in $G$ similarly:
$$\Nhood[k][G]{S} = G[\nhood[k][G]{S}].$$

We illustrate these definitions in \autoref{fig:vizinhanca}.
Note that $\nhood[k]{S} \subseteq \nhood[k+1]{S}$
and that $\Nhood[k]{S}$ is an induced subgraph of $\Nhood[k+1]{S}$.
We also note that, in contexts such as mathematical sociology,
the closed neighborhood of a vertex is known as
an \emph{egonet}~\cite{SocNetAnalysis,NetworkParadigm}.

\begin{figure}[ht]
\centering
        \begin{tikzpicture}[scale=0.9,node distance=3em,every node/.style={circle,minimum size=1.5em,inner sep=2}]
        \node[draw=red!70!black, very thick, fill=red!30] at (0,0) (v) {$v$};
        \node[draw=red!70!black, very thick, fill=red!30, above of=v] (w) {$w$};
        \node[draw=red!70!black, very thick, fill=red!30, above left of=v] (w1) {};
        \node[draw=red!70!black, very thick, fill=red!30, above right of=v] (w2) {};
        \node[draw=orange!70!black, fill=orange!30, left of=v] (w3) {};
        \node[draw=red!70!black, very thick, fill=red!30, right of=v] (w4) {};
        \node[draw=red!70!black, very thick, fill=red!30, below left of=v] (w5) {};
        \node[draw=orange!70!black, fill=orange!30, below right of=v] (w6) {};
        \node[draw=orange!70!black, fill=orange!30, above of=w] (x1) {};
        \node[draw=orange!70!black, fill=orange!30, above of=w1] (x2) {};
        \node[draw=orange!70!black, fill=orange!30, above left of=w1] (x3) {};
        \node[draw=orange!70!black, fill=orange!30, above of=w2] (x4) {};
        \node[draw=orange!70!black, fill=orange!30, above right of=w2] (x5) {};
        \node[draw=orange!70!black, fill=orange!30, above left of=w3] (x6) {};
        \node[draw=black, left of=w3] (x7) {};
        \node[draw=orange!70!black, fill=orange!30, below left of=w3] (x8) {};
        \node[draw=orange!70!black, fill=orange!30, above right of=w4] (x9) {};
        \node[draw=orange!70!black, fill=orange!30, right of=w4] (x10) {};
        \node[draw=black, below right of=w4] (x11) {};
        \node[draw=black, above of=x1] (y1) {};
        \node[draw=black, above of=x2] (y2) {};
        \node[draw=black, above left of=x2] (y3) {};
        \node[draw=black, above of=x4] (y4) {};
        \node[draw=black, above right of=x4] (y5) {};
        \node[draw=black, above left of=x3] (y6) {};
        \node[draw=black, left of=x3] (y7) {};
        \node[draw=black, above right of=x5] (y8) {};
        \node[draw=black, right of=x5] (y9) {};
        \node[draw=black, left of=x6] (y10) {};
        \node[draw=black, right of=x9] (y11) {};
        \node[draw=black, left of=x8] (y12) {};
        \node[draw=black, right of=x11] (y13) {};
        
        \path[draw=red!70!black, very thick] (v) -- (w) -- (w1) -- (v) -- (w2) -- (w4) -- (v) -- (w5) (w1) -- (w4) -- (w5) (w)--(w2);
        \path[draw=orange!70!black] (w1) -- (w3) -- (w5) (w2) -- (w3) -- (w6) -- (w4);
        \path[draw=orange!70!black] (x1) -- (w) -- (x2) -- (w1) -- (x3) -- (x2) -- (x1) -- (x4) -- (x5) -- (w2) -- (x4) -- (w) (x3) -- (x6) -- (w3) (w5) -- (x8) (x5) -- (x9) -- (w4) -- (x10) (x4) -- (x6) -- (w1) (x2) -- (x9) -- (w2) (w3) -- (x3) (w4) -- (x5);
        
        \path[draw=black, densely dotted] (w3) -- (x7) (w6) -- (x11) (y1) -- (x1) -- (y2) -- (x2) -- (y3) -- (x3) -- (y6) -- (y3) -- (y2) -- (y1) -- (y4) -- (y5) -- (y8) -- (x5) -- (y5) -- (x4) -- (y4) -- (x1) (y6) -- (y7) -- (x3) -- (y10) -- (x6) -- (y7) -- (y10) -- (x7) -- (x8) -- (y12) -- (x7) (y8) -- (y9) -- (x5) -- (y11) -- (x9) -- (y9) -- (y11) -- (x10) -- (x11) -- (y13) -- (x10);
        \end{tikzpicture}

\caption[Example of neighborhoods.]{ Examples of neighborhoods of a vertex in a graph. The subgraph highlighted in red is $\Nhood[1]{v} = \Nhood{v}$, and including the orange highlights we obtain $\Nhood[2]{v}$. The vertex sets of these subgraphs are, respectively, $\nhood[1]{v} = \nhood{v}$ and $\nhood[2]{v}$. For $k \geq 3$, $\Nhood[k]{v}=G$.}\label{fig:vizinhanca}
\end{figure}
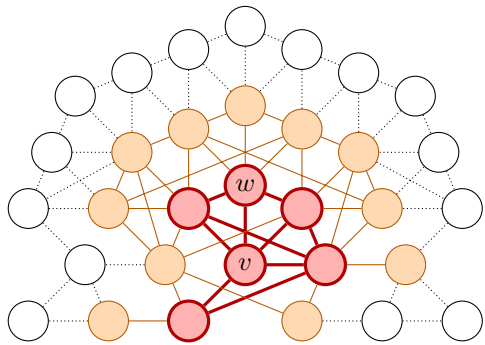

For the purposes of this work, we use closed $k$-neighborhoods
around single vertices as a proxy for ``locality''.
This motivates the following definition:

\begin{definition}
\label{simet-loc-vertices}
Given a graph $G=(V,E)$,
two vertices $v_1,v_2\in V$ are \emph{$k$-locally symmetric}
if there is an isomorphism $f$
between $\Nhood[k]{v_1}$ and $\Nhood[k]{v_2}$
such that $f(v_1) = v_2$.
\end{definition}

Therefore, two vertices $v_1$ and $v_2$ are $k$-locally symmetric
if the $k$-th order local structures in which $v_1$ and $v_2$ are located
are equivalent, with $v_1$ and $v_2$ equivalently located in these structures.
If $k = 1$, we will briefly say that the vertices are \emph{locally symmetric}.

One interesting feature of $k$-local symmetry
is that it naturally leads to the construction of a symmetry hierarchy,
which includes global symmetry as the most restrictive one,
as evinced by the following result.

\begin{proposition}
\label{hierarquia}
Let $v_1$ and $v_2$ be vertices of $G=(V,E)$, and let $k\in\Nat$.
Then the following statements hold:
\begin{enumerate}
\item If $v_1$ and $v_2$ are $(k+1)$-locally symmetric,
then $v_1$ and $v_2$ are $k$-locally symmetric;
\item If $k \geq \diam{G}$, then $v_1$ and $v_2$ are $k$-locally symmetric
if and only if $v_1$ and $v_2$ are globally symmetric.
\end{enumerate}
\end{proposition}

\begin{proof}
If $v_1$ and $v_2$ are $(k+1)$-locally symmetric,
then there is an isomorphism $f$
between $\Nhood[k+1]{v_1}$ and $\Nhood[k+1]{v_2}$ with $f(v_1) = v_2$.
Since isomorphisms preserve distances, for any $d\in\Nat$,
two vertices $u$ and $v$ of $\Nhood[k+1]{v_1}$ are at distance $d$
if and only if $f(u)$ and $f(v)$ are at distance $d$ in $\Nhood[k+1]{v_2}$.
Therefore, for every vertex $w$ in $\Nhood[k+1]{v_2}$,
it is true that $w\in\Nhood[k]{v_1} \iff f(w)\in\Nhood[k]{v_2}$.
This means that $f\on{\Nhood[k]{v_1}}$ is a bijection
between $\nhood[k]{v_1}$ and $\nhood[k]{v_2}$.
Note that, since $f$ also preserves edges, so does $f\on{\Nhood[k]{v_1}}$,
thus $f\on{\Nhood[k]{v_1}}$ is an isomorphism
between $\Nhood[k]{v_1}$ and $\Nhood[k]{v_2}$.
Since $f\on{\Nhood[k]{v_1}}(v_1) = f(v_1) = v_2$,
it follows that $v_1$ and $v_2$ are, by definition, $k$-local symmetric.

This proves the first statement.
The second statement follows from the fact that,
if $k\geq \diam{G}$, then any two vertices of $G$
are at distance $k$ or smaller from each other.
Therefore, $\nhood[k]{v} = V$ and $\Nhood[k]{v} = G$,
and the definitions of $k$-local symmetry and global symmetry are equivalent.
\end{proof}

Finally, we define local symmetry in graphs:

\begin{definition}
\label{simet-loc-1grafo}
Let $G=(V,E)$ be a graph. Then $G$ is \emph{$k$-locally symmetric}
if there are $u,v\in V$ distinct and $k$-locally symmetric.
\end{definition}

Note that this definition for $k$-local symmetry is analogous
to \autoref{simet-gl-1grafo} for global symmetry.
Therefore, our definitions of $k$-local symmetry ---
for two vertices and for a single graph ---
are mutually consistent in the same fashion as those of global symmetry.
This definition also implies a symmetry hierarchy
similar to \autoref{hierarquia}:

\begin{proposition}
\label{hierarquia-1grafo}
Let $G=(V,E)$ be a graph, and let $k\in\Nat$.
Then the following statements hold:
\begin{enumerate}
\item If $G$ is $(k+1)$-locally symmetric, then $G$ is $k$-locally symmetric;
\item If $k \geq \diam{G}$, then $G$ is $k$-locally symmetric
if and only if $G$ is globally symmetric.
\end{enumerate}
\end{proposition}

Given our interest in understanding symmetry in random network models,
it is natural that we start our analysis with the $G(n,p)$ model.
In particular, we would like to determine
if there is any fundamental difference
between the emergence of local symmetry and global symmetry in this model.
It is known~\cite{AsymGraphs}~\cite[p.230]{Bollobas-RandomGraphs} that,
for the similar $G(n,m)$ model of random graphs,
an asymmetric graph is obtained a.a.s. if and only if
$2m/\binom{n}{2} \geq \log n + \omega(1)$ and
$n-1 - 2m/\binom{n}{2} \geq \log n + \omega(1)$
--- that is, when the model exhibits average degree
at least slightly larger than $\log n$ and
at most slightly smaller than $n-1-\log n$.
For $G(n,p)$, an analogous but slightly weaker result is known:
the $G(n,p)$ random graph is asymmetric a.a.s.
if $p \in [\log n/n, 1-\log n/n]$~\cite{AsymRandRegGraphs},
and it is globally symmetric if $p \ll \log n/n$
(due to the existence of isolated vertices a.a.s.~\cite{NooC})
or $1-p \ll \log n/n$ (due to the existence of universal vertices a.a.s.).

\section{Symmetry regimes}

Our first result is the identification
of a local symmetry regime for $G(n,p)$:

\begin{theorem}
\label{1-loc-sym}
A $G(n,p)$ random graph, with $p = o(n^{-2/3})$,
is locally symmetric a.a.s.
\end{theorem}


Let us proceed with some terminology before the proof of this statement.
Recall the definition of closed 1-neighborhood of a vertex.
We call this vertex the \emph{center} of this subgraph,
with all edges between the center and other vertices
said to be \emph{core edges},
and all remaining edges termed \emph{peripheral edges}.

The idea behind the proof is that peripheral edges always close triangles,
but in this regime, $G(n,p)$ does not have too many triangles,
so the closed neighborhood of many vertices are simple stars.
Two such vertices will have isomorphic closed neighborhoods
simply by having the same degree, which will happen for some pair,
since the degrees in $G(n,p)$ concentrate heavily around their mean.

To execute the proof, we will need two auxiliary results:

\begin{lemma}
\label{concentrated-degree}
Let $G = (V,E)$ be a $G(n,p)$ random graph.
If $p = \omega(\log n/n)$, then for any fixed $\delta\in(0,1)$,
the degree of all vertices of $G$
are within the range $(n-1)p(1\pm\delta)$, a.a.s.
\end{lemma}

\begin{proof}
Let $d_v$ be the degree of vertex $v$.
We know that $d_v \eqD \Bin(n-1,p)$.
By the Chernoff bound, for any $\delta\in (0,1)$:
$$\Prob[d_v\notin(n-1)p(1\pm\delta)] \leq 2e^{-(n-1)p\delta^2/3}.$$

The union bound implies:
$$\Prob[\exists~v~:~d_v\notin(n-1)p(1\pm\delta)] \leq 2ne^{-(n-1)p\delta^2/3}$$
and, since $np = \omega(\log n)$ by hypothesis,
the right-hand side is $2ne^{-\omega(\log n)} = 2no(1/n) = o(1)$.
\end{proof}

\begin{lemma}
\label{few-triangles}
Let $G = (V,E)$ be a $G(n,p)$ random graph,
and let $T$ be the number of triangles in $G$.
Then, $\Exp[T] = \binom{n}{3}p^3$ and, if $p = \omega(1/n)$,
$\Prob[|T-\Exp[T]|<c\cdot\Exp[T]]\to 1$ for any fixed $c>0$.
\end{lemma}

\begin{proof}
We denote by $\binom{V}{3}$ the set of unordered triples of vertices,
and for each triple $t = (i,j,k)$, we define the event
\{$t$ is $\Delta$\} = \{$(i,j),(i,k),(j,k)\in E$\}.
Then:
$$T = \sum_{t\in\binom{V}{3}}\Ind_{\{t\text{ is }\Delta\}}.$$

It easily follows from linearity of expectation and independence of edges that
$\Exp[T] = \sum_{t\in\binom{V}{3}}\Prob[t\text{ is }\Delta] = \binom{n}{3}p^3$.

We also need to estimate the variance of $T$, which we denote by $\Var[T]$.
For such, we need an expression for its second moment:

\begin{align*}
\Exp[T^2] &= \Exp\left[
\left(\sum_{t\in\binom{V}{3}}\Ind_{\{t\text{ is }\Delta\}}\right)
\left(\sum_{t\in\binom{V}{3}}\Ind_{\{t\text{ is }\Delta\}}\right)\right] \\
&= \sum_{t,t'\in\binom{V}{3}}\Exp[
\Ind_{\{t\text{ is }\Delta\}}\Ind_{\{t'\text{ is }\Delta\}}] \\
&= \sum_{t,t'\in\binom{V}{3}}\Prob[t\text{ is }\Delta,t'\text{ is }\Delta]
\end{align*}

This summation can be broken into four pieces,
based on the relationship between the two triples of vertices, $t$ and $t'$:

\begin{description}
\item[No common vertices] All edges of $t$ are independent of all edges of $t'$,
so $\Prob[t\text{ is }\Delta,t'\text{ is }\Delta] = p^6$;
\item[One common vertex] Again, edges of $t$ are independent of edges of $t'$,
and $\Prob[t\text{ is }\Delta,t'\text{ is }\Delta] = p^6$;
\item[Two common vertices] $t$ and $t'$ share the edge between common vertices,
comprising a total of 5 edges ---
thus, $\Prob[t\text{ is }\Delta,t'\text{ is }\Delta] = p^5$;
\item[Three common vertices] In this case, $t = t'$ and
$\Prob[t\text{ is }\Delta,t'\text{ is }\Delta] = p^3$.
\end{description}

We must also count how many triples fit into each of these cases:
\begin{description}
\item[No common vertices] $\binom{n}{3}\binom{n-3}{3}$ triples;
\item[One common vertex] $\binom{n}{3}\cdot 3\binom{n-3}{2}$ triples;
\item[Two common vertices] $\binom{n}{3}\cdot 3(n-3)$ triples;
\item[Three common vertices] $\binom{n}{3}$ triples.
\end{description}

We can now calculate the second moment of $T$:

\begin{align*}
\Exp[T^2]
&= \sum_{t,t'\in\binom{V}{3}}\Prob[t\text{ is }\Delta,t'\text{ is }\Delta] \\
&= \binom{n}{3}p^3 + 3\binom{n}{3}(n-3)p^5 + 3\binom{n}{3}\binom{n-3}{2}p^6
+ \binom{n}{3}\binom{n-3}{3}p^6 \\
&= \binom{n}{3}p^3
\left[1 + 3np^2 - 9p^2 + \frac{3}{2}n^2p^3
- \frac{3}{2}7np^3 + \frac{3}{2}12p^3\right.\\
&\qquad+ \frac{1}{6}n^3p^3 - \frac{1}{6}12n^2p^3
+ \left.\frac{1}{6}47np^3 - \frac{1}{6}60p^3\right]. \\
\end{align*}

We can also obtain an expression for $\Exp[T]^2$:

\begin{align*}
\Exp[T]^2
&= \left(\binom{n}{3}p^3\right)^2 \\
&= \binom{n}{3}p^3
\left[\frac{1}{6}n^3p^3-\frac{1}{6}3n^2p^3+\frac{1}{6}2np^3\right].
\end{align*}

Combining these expressions yields an expression for $\Var[T]$:

\begin{align*}
\Var[T] &= \Exp[T^2] - \Exp[T]^2 \\
&= \binom{n}{3}p^3(1 + 3np^2 - 9p^2 - 3np^3 + 8p^3)
\end{align*}

Note that, if $np = \omega(1)$,
then $\Var[T]/\Exp[T]^2 \to 0$ when $n\to\infty$.
Chebyshev's inequality states that, for any $k > 0$,
$$\Prob[|T-\Exp[T]|\geq k\sqrt{\Var[T]}] \leq \frac{1}{k^2}.$$
Then, for any $c > 0$, we can take $k = c~\Exp[T]/\sqrt{\Var[T]}$ and obtain
$$\Prob[|T-\Exp[T]|\geq c~\Exp[T]] \leq \frac{\Var[T]}{c^2\Exp[T]^2},$$
which vanishes for fixed $c$.
\end{proof}

We can now proceed to the proof of \autoref{1-loc-sym}.

\begin{proof}[Proof of \autoref{1-loc-sym}]
We will first prove the result under the additional assumption that
$p = \omega(\log n/n)$, which will be removed by the end of the proof.

Let $G = (V,E)$ be a $G(n,p)$ random graph,
and let $T$ be the number of triangles in $G$.
Fix $c>0$ arbitrary and $\delta \in (0,1)$,
and define the following sequences of events on $G$:
$$A_n = \{|T-\Exp[T]|<c~\Exp[T]\},$$
$$B_n = \{\text{all degrees }\leq(n-1)p(1+\delta)\}.$$

\autoref{concentrated-degree} and \autoref{few-triangles} ensure that
$\Prob(A_n\cap B_n)\to 1$ as $n\to\infty$.
We will prove that this intersection event
is contained in the event \{$G(n,p)$ is locally symmetric\}.

In $A_n$, there are at most $(1+c)\binom{n}{3}p^3$ triangles in $G$.
Each edge in a triangle appears as a peripheral edge
in the closed neighborhood of its opposite vertex in this triangle.
Therefore, summing over all vertices' closed neighborhoods,
there are, at most, $3(1+c)\binom{n}{3}p^3 = o(n)$ peripheral edges.
This implies that, at least, $n - 3(1+c)\binom{n}{3}p^3$ vertices
have no peripheral edges in their closed neighborhoods.
Let $C$ be the set of such vertices.

In $B_n$, every vertex has degrees in the range $[0,(n-1)p(1\pm\delta))$.
Let $D$ be the set of integers satisfying this property.

Now, for $p = o(n^{-2/3})$, we have:
\begin{align*}
|C| &= n - \theta((np)^3) \\
& = n - o(n), \\
|D| &\leq (n-1)p(1+\delta) + 1 \\
&= o(\sqrt[3]{n}). \\
\end{align*}

This implies that $|C| > |D|$ for sufficiently large $n$.
By the pigeonhole principle, in $A_n\cap B_n$,
there must be at least two vertices in $C$ with the same degree.
These vertices' closed neighborhoods are stars of the same size,
therefore they must be isomorphic.
This implies our result.

It still remains for us to lift the assumption that $p = \omega(\log n/n)$.
The main issue to resolve is that we cannot bound directly the size of $D$,
since \autoref{concentrated-degree} does not apply.
Instead, we must try and sidestep the issue to make the argument work.

Pick some $p' \geq p$, such that
$\omega(\log n/n) \leq p' \leq o(n^{-2/3})$,
and consider the sequence of events:
$$B'_n = \{\text{all degrees }\leq(n-1)p'(1+\delta)\}.$$

Note that our probability measure $\Prob$
is associated with $G(n,p)$, not with $G(n,p')$.
However, since $B'_n$ happens a.a.s. under $G(n,p')$
(by \autoref{concentrated-degree}),
and the events $B'_n$ represent a decreasing property\footnote{A
\emph{decreasing property} is a property preserved
under removal of edges (such as ``$G(n,p)$ is not connected'').
A standard coupling argument shows that, for all $n$,
the probability of such properties is a decreasing function of $p$.},
it must also happen a.a.s. under $G(n,p)$, so $\Prob(B'_n) \to 1$.

Replacing $B_n$ by $B'_n$ in our previous argument,
the conclusion again follows.

\end{proof}

\section{Degree function}

For the identification of asymmetry regimes, we need several additional tools.
The core concept is that of the degree sequence of a graph,
which we present in a slightly different form:

\begin{definition}
For any graph $G=(V,E)$, the \emph{degree function} of $G$ is
the function $\func{\df{G}}{\Nat_0}{\Nat_0}$ such that
$\df{G}(k) = |\{v\in V~:~d_G(v)=k\}|$ for all $k\in\Nat_0$.
\end{definition}

The degree function of $G$ simply returns, for an input $k$,
the number of vertices with degree $k$ in $G$.
This makes it equivalent to the degree sequence of $G$,
whenever the listing order of the degrees is irrelevant.

To identify asymmetry regimes, we must identify conditions under which
no two vertices in a $G(n,p)$ random graph are symmetric a.a.s.
The asymmetry of vertices is defined as
the lack of an isomorphism between their closed neighborhoods,
which relates to degree functions via the following remark:

\begin{remark}
\label{iso-df}
For $G$ and $G'$ isomorphic graphs, $\df{G} \equiv \df{G'}$.
\end{remark}


To enable a more fine-grained look into these closed neighborhoods,
we need one additional definition:

\begin{definition}
Let $G=(V,E)$ and $G'=(V',E')$ be two graphs.
The \emph{degree sequence edit distance} between $G$ and $G'$
(denoted by $\dsed(G,G')$)
is given by: $$\dsed(G,G') = \sum_k |\df{G}(k)-\df{G'}(k)|.$$
\end{definition}

\begin{remark}
\label{dsed-df}
Let $\mu$ be the counting measure on $\Nat$. Then:
$$\dsed(G,G') = \int |\df{G}-\df{H}|~\dd\mu = \|\df{G}-\df{H}\|_1.$$

Thus $\dsed$ is a semimetric over the space of all graphs,
with $\dsed(G,G') = 0$ iff $\df{G} \equiv \df{G'}$.
\end{remark}


\begin{remark}
For $G$ and $G'$ isomorphic graphs, $\dsed(G,G') = 0$.
\end{remark}

To better understand the meaning of this edit distance,
consider any partial mapping between the vertex sets of two graphs.
We can measure the \emph{degree mismatch count} of this mapping,
which is a simple count of vertices, from both graphs,
that are either mapped to vertices with different degrees or left unmapped.

\begin{definition}
Let $G=(V,E)$, $G'=(V',E')$ be two graphs,
and let $\func{f}{S}{S'}$ be a bijective function
from $S\subseteq V$ to $S'\subseteq V'$.
The \emph{degree mismatch count} of $f$ (denoted by $\delta_f$) is given by
\begin{align*}
\delta_f &= |\{v\in S~:~d_G(v)\neq d_{G'}(f(v))\}|
+ |V\setminus S| \\
&\qquad+ |\{v'\in S'~:~d_{G'}(v')\neq d_G(f^{-1}(v'))\}| + |V'\setminus S'|.
\end{align*}
\end{definition}

The degree sequence edit distance between two graphs is, then,
the smallest possible degree mismatch count between their vertex sets:

\begin{theorem}
\label{dsed-delta}
Let $G=(V,E)$ and $G'=(V',E')$ be two graphs. Then:
$$\dsed(G,G') = \min_{ \substack{ \func{f}{S}{S'}\\f\text{ \rm  bijective},
S\subseteq V, S'\subseteq V'} }\delta_f.$$
\end{theorem}

\begin{proof}
The statement follows from inequalities on both directions.
We begin by showing $\dsed(G,G') \geq \min_f \delta_f$,.
It is enough to construct a function $g$ such that $\dsed(G,G') = \delta_g$,
and we perform this construction by ``slices'', one for each possible vertex degree.

For each $k\in\Nat_0$, let
$v^k_1,\dots,v^k_{\df{G}(k)}$ and $w^k_1,\dots,w^k_{\df{G'}(k)}$
be enumerations of degree-$k$ vertices in $G$ and $G'$, respectively.
Write $m_k = \min(\df{G}(k),\df{G'}(k))$,
$V_k = \{v^k_1,\dots,v^k_{m_k}\}$ and $V'_k = \{w^k_1,\dots,w^k_{m_k}\}$,
and construct function $\func{g_k}{V_k}{V'_k}$
mapping $v^k_j$ to $w^k_j$, for $j = 1,\dots,m_k$.
Note that $g_k$ leaves $|\df{G}(k)-\df{G'}(k)|$ degree-$k$ nodes unmapped,
all from $V$ (if $\df{G}(k)\geq\df{G'}(k)$)
or from $V'$ (if $\df{G}(k)\leq\df{G'}(k)$),
and which, by construction, cannot be mapped by any other function $g_{k'}$.

Now, construct function $\func{g}{\cup_k V_k}{\cup_k V'_k}$ as $g = \cup_k g_k$.
By double counting the number of nodes left unmapped by $g$
(from both $V$ and $V'$),
we see that this number is equal
to $|V\setminus(\cup_k V_k)|+|V'\setminus(\cup_k V'_k)|$ by definition,
and to $\sum_k |\df{G}(k)-\df{G'}(k)|$ by construction.
Furthermore, our construction also ensures that
for every node mapped by $g$ with degree $k$ in $G$,
its image has degree $k$ in $G$, and vice-versa for nodes in $G'$.
Therefore, it holds that
$|\{v\in \cup_k V_k~:~d_G(v)\neq d_{G'}(g(v))\}| =
|\{v\in \cup_k V'_k~:~d_{G'}(v)\neq d_G(g^{-1}(v))\}| = 0$, and:
\begin{align*}
\dsed(G,G') &= \sum_k |\df{G}(k)-\df{G'}(k)| \\
&= \sum_k |\df{G}(k)-\df{G'}(k)| + 0 \\
&= |V\setminus(\cup_k V_k)|+|V'\setminus(\cup_k V'_k)| \\
&\qquad+ |\{v\in \cup_k V_k~:~d_G(v)\neq d_{G'}(g(v))\}| \\
&\qquad+ |\{v'\in \cup_k V'_k~:~d_{G'}(v')\neq d_G(g^{-1}(v'))\}|\\
&= \delta_g.
\end{align*}

Now, it only remains to show $\dsed(G,G') \leq \min_f \delta_f$.
We will show that $\dsed(G,G') \leq \delta_h$ for any partial mapping $h$,
and we will again proceed by ``slices'' in our argument.
Let $\func{h}{S}{S'}$ be an arbitrary bijection
with $S\subseteq V$ and $S'\subseteq V'$.

Consider all vertices with degree $k$ in both $G$ and $G'$.
If $\df{G}(k)\geq\df{G'}(k)$, then
at most $\df{G'}(k)$ $k$-degree vertices in $G$
can be mapped by $h$ into $k$-degree nodes in $G'$.
This implies that at least $\df{G}(k)-\df{G'}(k)$ $k$-degree vertices in $G$
must either be left unmapped by $g$ (thus belonging to $V\setminus S$)
or be mapped to vertices in $G'$ with degree different than $k$
(and therefore belonging to $\{v\in \cup_k V_k~:~d_G(v)\neq d_{G'}(h(v))\}$).
Analogously, if $\df{G}(k)\leq\df{G'}(k)$,
at least $\df{G'}(k)-\df{G}(k)$ $k$-degree vertices in $G'$
must either be left unmapped by $g$ (this time, belonging in $V'\setminus S'$)
or be mapped to vertices in $G$ with degree different than $k$.

In both cases, there is a contribution of $|\df{G}(k)-\df{G'}(k)|$ to $\delta_h$
coming exclusively from $k$-degree nodes in $G$ and $G'$.
Putting together contributions from all ``slices''
ensures that $\delta_h\geq\sum_k|\df{G}(k)-\df{G'}(k)| = \dsed(G,G')$.
Since $h$ was arbitrary, $\min_h \delta_h \geq\dsed(G,G')$, as desired.
\end{proof}

This property can be used to relate
the degree functions of a graph and its subgraphs.
For any set of vertices $S\subseteq V$,
denote by $G[S]$ the subgraph of $G$ induced by $S$, 
and by $C(S)$ the set of edges with
one endpoint in $S$ and another in $V\setminus S$.

\begin{corollary}
\label{dsed-subgraph}
For any graph $G=(V,E)$ and any $S\subset V$,
$$\dsed(G,G[S]) \leq |V\setminus S| + |C(S)|.$$
\end{corollary}

\begin{proof}
By virtue of \autoref{dsed-delta}, it is enough to show that
$\delta_g = |V\setminus S| + |C(S)|$
for some partial mapping $g$ between $G$ and $G[S]$.

Take $\func{g}{S}{S}$ to be the identity function on $S$.
Then $g$ is a partial mapping between the vertex sets of $G$ and $G[S]$,
and its degree mismatch count is given by:
$$\delta_g = |\{v\in G[S]~:~d_{G[S]}(v)\neq d_G(v)\}| + |V\setminus S|.$$

Now, notice that, for any vertex $v\in S$, $d_{G[S]}(v)\neq d_G(v)$
if and only if $v$ is adjacent to some vertex outside $S$.
Since there are $|C(S)|$ edges between $S$ and $V\setminus S$,
there can be at most $|C(S)|$ such vertices, and the result follows.
\end{proof}

\section{Asymmetry regimes}
We can now proceed to the identification of a local asymmetry regime for
the $G(n,p)$ random graph model.

\begin{theorem}
\label{asym}
A $G(n,p)$ random graph with
$\omega(n^{-1/2+\delta_1}) \leq p \leq o(n^{-3/7-\delta_2})$
for constant $\delta_1,\delta_2>0$
is locally asymmetric a.a.s.
\end{theorem}

Once again, we present some intermediate results
before proceeding to the proof of this result.

\begin{lemma}
\label{rev-chernoff}
Let $X$ be a $\Bin(n,p)$ random variable with $p < 1/2$.
If $\varepsilon > 0$ and $np\varepsilon^2\geq 3$, then
$$\Prob[|X-np|\geq \varepsilon np] \geq 2\exp\{-9np\varepsilon^2\}.$$
\end{lemma}
\begin{proof}
See Lemma 5.2 of~\cite{Klein15Number}.
\end{proof}

\begin{lemma}
\label{mode-multinom}
Let $\vec{X}$ be a multinomial random vector
$\Mult(n,p_1,\dots,p_k)$ with $k$ fixed,
and let $0 < \beta < 1$ be also fixed.
If $\Omega(n^{-\beta})\leq p_1,\dots,p_{k-1} \leq o(1)$, then
$$\max_{\vec{x}}\Prob[\vec{X}=\vec{x}] = O(n^{-(k-1)(1-\beta)/2}).$$
\end{lemma}

\begin{proof}
Let $x^\ast = (x^\ast_1,\dotsc, x^\ast_k)$ be
the mode of $\Mult(n,p_1,\dotsc, p_k)$.
It is known~\cite{ModeMultinomial} that $x^\ast_i = I(np_i)$,
where $I(a)$ is either $\lfloor a\rfloor$ or $\lceil a-1 \rceil$.
This implies that $x^\ast_i \leq np_i \leq x^\ast_i + 1$.
Also, since $\beta < 1$, it holds that $np_i\to\infty$ for all $i$,
which implies that, for large enough $n$, $x^\ast_i \geq 1$ for all $i$.

Using these inequalities, Stirling's approximation,
and known bounds for the exponential function,
we have, for large enough $n$:

\begin{align*}
\max_{\vec{x}} \Prob[\vec{X}=\vec{x}] &= \Prob[\vec{X}=x^\ast] \\
&= \frac{n!}{\prod_{i=1}^n x^\ast_i!}
\prod_{i=1}^n p_i^{x^\ast_i} \\
&\leq \frac{en^ne^{-n}\sqrt{n}}{(\sqrt{2\pi})^k
\prod_{i=1}^k (x^\ast_i)^{x^\ast_i}\sqrt{x^\ast_i}e^{-x^\ast_i}}
\prod_{i=1}^k p_i^{x^\ast_i} \\
&\leq \frac{e\sqrt{n}}{(\sqrt{2\pi})^k}
\prod_{i=1}^k\left(\frac{np_i}{x^\ast_i}\right)^{x^\ast_i}
\frac{1}{\sqrt{x^\ast_i\dotsm x^\ast_{k-1}}}\frac{1}{\sqrt{np_k}} \\
&\leq \frac{e}{(\sqrt{2\pi})^k}
\prod_{i=1}^k\left(1+\frac{1}{x^\ast_i}\right)
\frac{1}{\sqrt{np_1\dotsm np_{k-1}}}\frac{1}{\sqrt{p_k}} \\
&\leq \frac{e}{(\sqrt{2\pi})^k}\cdot e^k \cdot
\frac{1}{\sqrt{np_1\dotsm np_{k-1}}}\frac{1}{\sqrt{p_k}}. \\
\end{align*}

Now, the following inequalities hold for large enough $n$.
First, by hypothesis, for every $i\leq k-1$,
we have $p_i \geq c_i n^{-\beta}$ for some constant $c_i$.
Second, $p_k \geq c_k$ for some constant $c_k$,
since the hypotheses imply that $p_k\to 1$.
Thus:

\begin{align*}
\max_{\vec{x}} \Prob[\vec{X}=\vec{x}]
&\leq \frac{e}{(\sqrt{2\pi})^k}\cdot e^k \cdot
\frac{1}{\sqrt{np_1\dotsm np_{k-1}}}\frac{1}{\sqrt{p_k}} \\
&\leq \frac{e^{k+1}}{(\sqrt{2\pi})^k}
\frac{1}{\sqrt{c_1\dotsm c_{k-1}}\cdot c_k}
\frac{1}{(\sqrt{n^{1-\beta}})^{k-1}} \\
&= Kn^{-(k-1)(1-\beta)/2}
\end{align*}
for some constant $K$. This concludes the proof.
\end{proof}

\begin{theorem}[Theorem 3.2 in~\cite{TwoRandomGraphs}]
\label{deg-seq-approx}
Let $\vec{D}_n^{(1)}$ and $\vec{D}_n^{(2)}$ be the degree sequences
of two independent $G(n,p)$ random graphs,
with probability distribution $\Prob_{\mathcal{D}_n}$.

Furthermore, let $\mathcal{F}_n$ be
the $\sigma$-algebra generated by $\vec{D}_n^{(1)}$ and $\vec{D}_n^{(2)}$,
and let $\Prob_{\mathcal{B}_n}$ be a probability measure under which
$\vec{D}_n^{(1)}$ and $\vec{D}_n^{(2)}$ are random vectors
with $n$ independent coordinates, each having distribution $\Bin(n-1,p)$.

Then, for any sequence of events $A_n$
measurable under $\mathcal{F}_n^{\otimes 2}$
and any fixed $a>0$,
if $\Prob_{\mathcal{B}_n}^{\otimes 2}(A_n) = o(n^{-a})$,
then $\Prob_{\mathcal{D}_n}^{\otimes 2}(A_n) = o(n^{-a})$.
\end{theorem}

In a nutshell, \autoref{deg-seq-approx} allows us to
consider the degree sequence of two $G(n,p)$ random graphs
as sequences of independent random variables,
without interfering with power-law decays
in the probability of events on this model.

\begin{theorem}
\label{large-dsed}
Let $G_1$ and $G_2$ be independent $G(n,p)$ random graphs,
with $p$ satisfying $\omega(\log n/n) \leq p \leq o(n^{-1/2})$.
Then, for any $\varepsilon > 0$ and any $a > 0$,
$$\dsed(G_1,G_2) \geq n^{1/2-\varepsilon}$$
with probability $1-o(n^{-a})$.
\end{theorem}

\begin{proof}
This proof will proceed as follows.
Instead of looking at the whole degree sequences of $G_1$ and $G_2$,
we will group several ranges of degrees into ``buckets'',
according to their distances to the expected degree of $G(n,p)$.
This will allow us to bound the probability that
each vertex belongs to each bucket, using Chernoff-like bounds.
Considering the degrees as independent random variables characterizes
the bucketed degree sequences as multinomial random vectors.
Furthermore, the buckets themselves are carefully chosen so that
the distribution of these vectors is not too concentrated, i.e.,
the probability of their modes is large enough.
This means these two vectors will most likely be
far apart from each other in $L_1$-norm,
which is the desired result.

We begin by noting that $\Delta(G_1,G_2) = \int|\df{G_1}-\df{G_2}|~\dd\mu$
is a function of the degree sequences of $G_1$ and $G_2$.
Let $\df{G_1}',\df{G_2}'$ be the degree functions
obtained by approximating these degree sequences
by sequences of $n$ independent $\Bin(n-1,p)$ random variables.
By virtue of \autoref{deg-seq-approx}, it is enough to show that
$$\int|\df{G_1}'-\df{G_2}'|~\dd\mu \geq n^{1/2-\varepsilon}$$
with probability $1-o(n^{-a})$.

Fix real positive numbers $\alpha>\sqrt{27}$ and $\beta<\varepsilon$,
and choose a natural number $b > 2a/(\varepsilon-\beta)$.
Let the real positive intervals $S_1,\dots,S_b$ be defined as
$$S_i = [\alpha^{-(i+1)}\sqrt{(n-1)p}\cdot f(n,p),
\alpha^{-i}\sqrt{(n-1)p}\cdot f(n,p)),$$
where $f(n,p) = (\min\{\log n, (n-1)p\})^{1/4}$.
For ease the notation, let the set
$S_{b+1} = \Real_+\setminus\cup_{i=1}^b S_i$
contain the remainder of the positive real line.

We now use these sets to group the vertices
in both graphs into ``buckets'', according to their degrees,
with set $S_i$ indicating the set of allowed degrees
according to their distances to the average degree $(n-1)p$.
More formally, define the sets of integers $B_1,\dots,B_{b+1}$
as
$$B_i = \{k\in\Nat_0~:~|k-(n-1)p|\in S_i\}.$$

Note that $\cup_i B_i = \Nat_0$ and
$B_i\cap B_j = \emptyset$ whenever $i\neq j$.
Now, write $T^{(1)}_i = \int_{B_i}\df{G_1}'~\dd\mu$.
$T^{(1)}_i$ counts the number of vertices of $G_1$ with degrees in $B_i$.
Similarly, write $T^{(2)}_i = \int_{B_i}\df{G_2}'~\dd\mu$.
Then it holds that
\begin{align*}
\int|\df{G_1}'-\df{G_2}'|~\dd\mu
&= \sum_{i=1}^{b+1} \int_{B_i}|\df{G_1}'-\df{G_2}'|~\dd\mu \\
&\geq \sum_{i=1}^{b+1} \left|\int_{B_i}\df{G_1}'\dd\mu -
\int_{B_i}\df{G_2}'~\dd\mu\right| \\
&= \sum_{i=1}^{b+1} |T_i^{(1)}-T_i^{(2)}| = \|\vec{T}^{(1)}-\vec{T}^{(2)}\|_1,
\end{align*}
where $\vec{T}^{(j)} = (T_1^{(j)},\dots,T_{b+1}^{(j)})$.

This means that
\begin{align*}
\Prob\left[\int|\df{G_1}'-\df{G_2}'|~\dd\mu \geq n^{1/2-\varepsilon}\right]
&\leq\Prob\left[\|\vec{T}^{(1)}-\vec{T}^{(2)}\|_1
\geq n^{1/2-\varepsilon}\right]\\
&=\sum_{\vec{t}}\Prob\left[\vec{T}^{(1)}=\vec{t},
\|\vec{T}^{(2)}-\vec{t}\|_1\geq n^{1/2-\varepsilon}\right] \\
&=\sum_{\vec{t}}\Prob\left[\vec{T}^{(1)}=\vec{t}\right]
\Prob\left[\|\vec{T}^{(2)}-\vec{t}\|_1\geq n^{1/2-\varepsilon}\right] \\
&\leq\sum_{\vec{t}}\Prob\left[\vec{T}^{(1)}=\vec{t}\right] \max_{\vec{t}}
\Prob\left[\|\vec{T}^{(2)}-\vec{t}\|_1\geq n^{1/2-\varepsilon}\right] \\
&= \max_{\vec{t}}
\Prob\left[\|\vec{T}^{(2)}-\vec{t}\|_1\geq n^{1/2-\varepsilon}\right].
\end{align*}

Note that, for any $\vec{t}$,
the event $\{\|\vec{T}^{(2)}-\vec{t}\|_1\geq n^{1/2-\varepsilon}\}$
has at most $(2n^{1/2-\varepsilon}+1)^b$ elements:
each of the first $b$ coordinates of $\vec{T}^{(2)}$
must be at distance at most $n^{1/2-\varepsilon}$
from the corresponding coordinate of $\vec{t}$
(for a maximum of $2n^{1/2-\varepsilon}+1$ valid options),
and the last one is uniquely determined from the previous choices,
since the coordinates must sum up to $n$.
This means that:

\begin{align*}
\Prob\left[\int|\df{G_1}'-\df{G_2}'|~\dd\mu \geq n^{1/2-\varepsilon}\right]
&\leq \max_{\vec{t}}
\Prob\left[\|\vec{T}^{(2)}-\vec{t}\|_1\geq n^{1/2-\varepsilon}\right] \\
&\leq (2n^{1/2-\varepsilon}+1)^b
\max_{\vec{t}}\Prob\left[\vec{T}^{(2)}=\vec{t}\right].
\end{align*}

Now, the degree of every vertex in $G_1$ or $G_2$ must belong to some $B_i$,
and these degrees are deemed to be i.i.d.\@ random variables,
by our initial argument regarding $\df{G_1}'$ and $\df{G_2}'$.
This implies that $\vec{T}^{(1)},\vec{T}^{(2)}$ are
multinomial random vectors $\Mult(n,p_1,\dots,p_{b+1})$,
where $p_i = \Prob[\Bin(n-1,p) \in B_i]$ is
the probability that the degree of a vertex belongs to $B_i$.

At this moment, to apply \autoref{mode-multinom},
we would like to bound from both sides the value of $p_i$ (for $i\leq b$).
For an upper bound, an application of the Chernoff bound suffices:

\begin{align*}
p_i &= \Prob[\Bin(n-1,p)\in B_i] \\
&\leq \Prob[|\Bin(n-1,p)-(n-1)p|\geq\alpha^{-(i-1)}\sqrt{(n-1)p}\cdot f(n,p)]\\
&\leq 2\exp\{-f(n,p)^2\alpha^{-2(i+1)}/3\},
\end{align*}
which is $o(1)$, since $f(n,p)\to\infty$.

For a lower bound, we use the Chernoff bound
and an application of \autoref{rev-chernoff},
noting that $\alpha^{-2(i+1)}f(n,p)^2\geq 3$ for large enough $n$:
\begin{align*}
p_i &= \Prob[\Bin(n-1,p)\in B_i] \\
&= \Prob[|\Bin(n-1,p)-(n-1)p|\geq \alpha^{-(i-1)}\sqrt{(n-1)p}\cdot f(n,p)] \\
&\qquad- \Prob[|\Bin(n-1,p)-(n-1)p|\geq \alpha^{-i}\sqrt{(n-1)p}\cdot f(n,p)]\\
&\geq 2\exp\{-9\alpha^{-2(i+1)}f(n,p)^2\}-2\exp\{-\alpha^{-2i}f(n,p)^2/3\} \\
&= 2\exp\{-9\alpha^{-2(i+1)}f(n,p)^2\}
(1-\exp\{-\alpha^{-2(i+1)}f(n,p)^2\}^\gamma),
\end{align*}
where, in the last passage, we let $\gamma = \alpha^2/3-9 > 0$.
Note that $\gamma>0$ and $f(n,p)\to\infty$ imply that
the term inside the parentheses tends to 1.
The remaining exponential satisfies
\begin{align*}
n^\beta\exp\{-9\alpha^{-2(i+1)}f(n,p)^2\}
&\geq \exp\{\beta\log n-9\alpha^{-2(i+1)}\sqrt{\log n}\} \\
&= \exp\{\beta\log n(1-9\alpha^{-2(i+1)}(\log^{-1/2} n))\} \\
&= \omega(1),
\end{align*}
thus the expression on the right-hand side is $\omega(n^{-\beta})$.

This means that the conditions of \autoref{mode-multinom} are satisfied
for random vectors $\vec{T}^{(1)}$ and $\vec{T}^{(2)}$, with $k = b+1$.
Therefore, $\max_{\vec{t}}\Prob[\vec{T}^{(2)}=\vec{t}] = O(n^{-b(1-\beta)/2})$,
and, by the choices of $b$ and $\beta$,
\begin{align*}
\Prob\left[\int|\df{G_1}'-\df{G_2}'|~\dd\mu \geq n^{1/2-\varepsilon}\right]
&\leq (2n^{1/2-\varepsilon}+1)^b
\max_{\vec{t}}\Prob\left[\vec{T}^{(2)}=\vec{t}\right] \\
&= \theta(n^{b(1/2-\varepsilon)})O(n^{-b(1-\beta)/2}) \\
&= O(n^{-b(\varepsilon-\beta)/2}) \\
&= o(n^{-a}).
\end{align*}
\end{proof}

With this, we can proceed to the proof of \autoref{asym}.

\begin{proof}[Proof of \autoref{asym}]
Let $G = (V,E)$ be a $G(n,p)$ random graph.
By the union bound, it is enough to prove that
any two distinct vertices $v_1,v_2$
are locally symmetric with probability $o(n^{-2})$.

Let $v_1,v_2\in V$ be arbitrary distinct vertices, and denote
by $X_1$ the set of neighbors of $v_1$ that are not neighbors of $v_2$,
by $X_2$ the set of neighbors of $v_2$ that are not neighbors of $v_1$,
and by $Y$ the set of common neighbors of $v_1$ and $v_2$.
Additionally, denote by $C_1$ the number of edges between $X_1$ and $Y$,
and by $C_2$ the number of edges between $X_2$ and $Y$.
Our goal is to show that, with probability $o(n^{-2})$,
$\Nhood{v_1}=G[X_1\cup Y]$ and $\Nhood{v_2}=G[X_2\cup Y]$ are not isomorphic.
Note that these two graphs are not independent,
since they share $G[Y]$ as a subgraph.
Our goal is to show that, even if $G[Y]$ is removed from these two subgraphs,
their degree sequences have large enough edit distance, by \autoref{large-dsed},
that the effect of reinserting $G[Y]$,
which is bounded by \autoref{dsed-subgraph},
is not enough to make these degree sequences equal,
which implies $\Nhood{v_1}$ and $\Nhood{v_2}$ cannot possibly be isomorphic.

Before proceeding to this, we will need a few concentration bounds
that will help us carry out our proof.
First, note that $v_2\in X_1$ and $v_1\in X_2$
if and only if $u$ and $v$ are neighbors,
and that all other vertices belong to $X_1$, $X_2$ and $Y$
independently from each other,
with probabilities $p(1-p)$, $p(1-p)$ and $p^2$, respectively.
Thus, $|X_1|,|X_2|\leD \Bin(n-2,p)+1$ and $|Y|\eqD \Bin(n-2,p^2)$.
By a similar reasoning, given $X_1$, $X_2$, and $Y$, it holds that
$C_1\eqD\Bin(|X_1||Y|,p)$ and $C_2\eqD\Bin(|X_2||Y|,p)$.

Now, set $0 < \varepsilon_1 < 1$ constant, $\varepsilon_2 = \log n$, and
$\varepsilon_3 > 0$ constant.
Define the following events:
\begin{align*}
A_1 &= \{(n-2)p(1-p)(1-\varepsilon_1) < |X_1|,|X_2|
< (n-2)p(1-p)(1+\varepsilon_1)\}, \\
A_2 &= \{|Y| < (n-2)p^2(1+\varepsilon_2)\}, \\
A_3 &= \{C_1,C_2 < sp(1+\varepsilon_3)\}, \\
\end{align*}
where $s = (n-2)^2p^3(1-p)(1+\varepsilon_1)(1+\varepsilon_2)$.

We will show that $\Prob(A_1,A_2,A_3) \geq 1 - o(n^{-2})$.
First, by the Chernoff bound,
\begin{align*}
\Prob(\overline{A_1}) &\leq
2\exp\{-(n-2)p(1-p)\varepsilon_1\min(1,\varepsilon_1)/3\} \\
&\qquad + 2\exp\{-(n-2)p(1-p)\varepsilon_1\min(1,\varepsilon_1)/2\}, \\
\Prob(\overline{A_2}) &\leq
\exp\{-(n-2)p^2\varepsilon_2\min(1,\varepsilon_2)/3\}.
\end{align*}

Furthermore, in the event $A_1\cap A_2$,
it holds that $|X_1||Y|,|X_2||Y| \leq s$,
which implies $C_1,C_2\leD \Bin(s,p)$ and, again by the Chernoff bound:

$$\Prob(A_1,A_2,\overline{A_3}) \leq
2\exp\{-sp\varepsilon_3\min(1,\varepsilon_3)/3\}.$$

Note that the upper bounds for $\Prob(\overline{A_1})$, $\Prob(\overline{A_2})$
and $\Prob(A_1,A_2,\overline{A_3})$ are all $o(n^{-2})$,
since $(n-2)p \geq \omega(\log n)$ and
$(n-2)p^2\varepsilon_2 \geq \omega(1)$ by hypothesis,
and $sp \geq \omega(\log n)$ as a consequence.
Thus, $\Prob(A_1,A_2,A_3) \geq 1 - o(n^{-2})$.

We can now resume the main thread in our proof and show that,
conditional on $A_1\cap A_2\cap A_3$,
$v_1$ and $v_2$ are locally asymmetric with probability $1-o(n^{-2})$.
Recall that $v_1$ and $v_2$ are locally symmetric only if
$\dsed(\Nhood{v_1},\Nhood{v_2}) = 0$.
We can assume that $v_1$ and $v_2$ have the same degree ---
i.e. $|\nhood{v_1}| = |\nhood{v_2}|$,
otherwise $\Nhood{v_1}$ and $\Nhood{v_2}$ have vertex sets of different sizes,
which implies $v_1$ and $v_2$ are locally asymmetric.
Note that this implies that $|X_1| = |X_2|$:
if $v_1$ and $v_2$ are adjacent, then
$\nhood{v_1} = X_1\uplus Y\uplus \{v_2\}$ and
$\nhood{v_2} = X_2\uplus Y\uplus \{v_1\}$;
if they are not adjacent, then
$\nhood{v_1} = X_1\uplus Y$ and $\nhood{v_2} = X_2\uplus Y$.

Let $N_1 = (V_1,E_1) = G[\nhood{v_1}\setminus\{v_1\}]$ and
$N_2 = (V_2,E_2) = G[\nhood{v_2}\setminus\{v_2\}]$, i.e.,
$\Nhood{v_1}$ and $\Nhood{v_2}$ with their centers removed.
Since $v_1$ and $v_2$ are universal vertices of their respective neighborhoods,
their removal keeps the edit distance between them unchanged, that is,
$$\dsed(\Nhood{v_1},\Nhood{v_2}) = \dsed(N_1,N_2).$$

Since $\dsed$ is a semimetric, it further holds that
$$\dsed(N_1,N_2) \geq \dsed(G[X_1],G[X_2])
- \dsed(G[X_1],N_1) - \dsed(G[X_2],N_2).$$

To lower bound the first term on the right-hand side,
we note that the existence of edges between vertices in $X_1$
happens with probability $p$
independently of the particular vertex pair,
thus $G[X_1]$ is a $G(|X_1|,p)$ random graph.
Since $p = \omega(n^{-1/2+\delta_1})$, by hypothesis,
and $|X_1| = \theta(np)$ in $A_1\cap A_2\cap A_3$,
it holds, for $n$ large enough and for constants $c,c'>0$, that
\begin{align*}
\frac{p}{\log |X_1|/|X_1|}
&\geq \frac{p\cdot cnp}{\log c'np} \\
&\geq \frac{\omega(n^{2\delta_1})}{\log c'n} \\
&\geq \omega(1).
\end{align*}

It also holds that $|X_1|p^2 = p\cdot\theta(np^3) = o(1)$.
Together, these inequalities imply
$\omega(\log|X_1|/|X_1|) \leq p \leq  o(|X_1|^{-1/2})$,
and therefore $G[X_1]$ satisfies the hypotheses of \autoref{large-dsed}.
An analogous argument implies that $G[X_2]$ also satisfies these hypotheses.

Furthermore, note that $G[X_1]$ and $G[X_2]$ are independent,
since their vertex sets are disjoint.
This allows us to apply the results of \autoref{large-dsed},
picking fixed $a=\frac{14}{1-14\delta_2}$ and
$\varepsilon < \frac{49\delta_2}{8-14\delta_2}$:
\begin{align*}
\Prob[\dsed(G[X_1],G[X_2]) \geq |X_1|^{1/2-\varepsilon}|A_1\cap A_2\cap A_3]
&= 1-o((|X_1|p)^{-a}) \\
&= 1-o((np^2)^{-a}) \\
&\geq 1-o(n^{-a(1+2(-3/7-\delta_2))}) \\
&= 1-o(n^{-a(1/7-2\delta_2)}) \\
&= 1-o(n^{-2}).
\end{align*}

Note that these choices of $a$ and $\varepsilon$ are always possible and valid,
since the hypotheses imply $\delta_2 < 1/14$,
which makes $\frac{49\delta_2}{8-14\delta_2} \geq 7\delta_2 > 0$
and $\frac{14}{4-7\delta_2} > 0$.
The two bounds on $|X_1|$ from the definition of event $A_1$
imply that, with probability $1-o(n^{-2})$,
$$\Delta(G[X_1],G[X_2]) \geq \Omega((np)^{1/2-\varepsilon}).$$

To upper bound the remaining two terms,
we begin by applying \autoref{dsed-subgraph}:
$$\dsed(G[X_1],N_1) \leq |V_1\setminus X_1| + |C(X_1)|.$$

Now, all vertices of $Y$ are in $V_1\setminus X_1$,
and so is $v_2$ if it is a neighbor of $v_1$ in $G$.
All other vertices of $G$ either are in $X_1$ or are not in $V_1$, thus
$|V_1\setminus X_1| \leq |Y| + 1$.
Furthermore, $|C(X_1)|$ counts all edges from $X_1$ to $V_1\setminus X_1$,
regardless of whether $v_2$ belongs to $V_1\setminus X_1$,
since no edge connects $X_1$ to $v_2$ by construction.
Thus, $|C(X_1)|$ counts edges from $X_1$ to $Y$, which implies $|C(X_1)| = C_1$.
It follows that

\begin{align*}
\dsed(G[X_1],N_1) &\leq |Y| + 1 + C_1 \\
&\leq (n-2)p^2(1+\varepsilon_2) +
(n-2)^2p^4(1-p)(1+\varepsilon_1)(1+\varepsilon_2)(1+\varepsilon_3) \\
&= \theta(n^2p^4\log n).
\end{align*}

By an analogous argument, the same inequality holds for $\dsed(G[X_2],N_2)$.
Now, note that

\begin{align*}
\frac{\dsed(G[X_1],N_1)+\dsed(G[X_2],N_2)}{\dsed(G[X_1],G[X_2])}
&\leq \frac{\theta(n^2p^4\log n)}{\Omega((np)^{1/2-\varepsilon})} \\
&= o(n^{3/2+\varepsilon}p^{7/2+\varepsilon}\log n) \\
&\leq o(n^{3/2+\varepsilon}n^{(-3/7-\delta_2)(7/2+\varepsilon)}\log n) \\
&= o(n^{(4-7\delta_2)\varepsilon/7-7\delta_2/2}\log n)
\end{align*}

The exponent of $n$ in this expression satisfies
\begin{align*}
\frac{(4-7\delta_2)\varepsilon}{7}-\frac{7\delta_2}{2}
&< \frac{4-7\delta_2}{7}\cdot\frac{49\delta_2}{8-14\delta_2}
-\frac{7\delta_2}{2} \\
&= \frac{7\delta_2}{2}-\frac{7\delta_2}{2} = 0,
\end{align*}
thus the left-hand side is $o(1)$.
As a consequence, $\dsed(\Nhood{v_1},\Nhood{v_2}) = \dsed(N_1,N_2)
\geq \dsed(G[X_1],G[X_2])(1+o(1))\geq\omega(1)$.
Therefore, conditional on $A_1\cap A_2\cap A_3$,
with probability $1-o(n^{-2})$, $v_1$ and $v_2$ are locally asymmetric.
\end{proof}

\section{Conclusion}

In this paper, we have introduced the concept of local symmetry in graphs,
and studied the asymptotic presence and absence of this property
in the Erdős-Rényi random graph model,
rigorously establishing regimes for either behavior to emerge a.a.s..
It is important to note that, in this model,
one can find locally symmetric graphs with high probability
even in regimes of unrealistically high average degree, close to $n^{1/3}$.
A natural question to ask is whether this behavior persists
when looking at $k$-local symmetry for $k>1$.
To answer this question would require a deeper understanding of
the combinatorial aspects of this problem.

One might also ask whether, and to which extent,
real-world networks exhibit local symmetry.
This would certainly depend on both
the nature of the network and its formation process.
However, for certain classes of networks,
the extant literature allows us to develop
some intuition about what answer to expect.
For instance, in the social network literature,
a number of recent works have attempted to explore
the limits of network
anonymization~\cite{ResistingStructReident}~\cite{DeAnonSocNets}.
In particular, a technique known as
\emph{percolation graph matching}, or PGM~\cite{PerformancePGM},
has been successfully used to match common nodes
in the Twitter and Flickr networks~\cite{DeAnonSocNets},
thus allowing knowledge of one network to be used to
break the anonymity of the other.
Intuitively, this suggests that each node in these social networks
can be uniquely identified structurally within them,
from which we would conclude that the networks are globally asymmetric.
However, the fact that the PGM technique works by exploiting local neighborhoods
would also indicate that these networks exhibit some kind of local asymmetry.
Further investigation of this matter would allow us to start exploring
the applicability of the local symmetry concept to real-world networks.

It should be noted that, unlike the $G(n,p)$ random graphs,
real-world networks exhibit much higher structural diversity of vertices.
For instance, networks such as the Internet~\cite{PowerLawInternetTopology},
the Web~\cite{GraphStructWeb},
and scientific collaboration networks~\cite{StructSciCollabNetworks}
are believed to have heavy-tailed degree distributions.
This diversity can lead to a similarly diverse behavior regarding
global and local symmetry.
It is known that most real-world networks are globally symmetric,
but the automorphism group of these networks is due to
a large number of small subgraphs which are themselves symmetric
and comprise vertices with small degrees~\cite{SymmComplexNetworks},
with high degree vertices being asymmetric to any other vertex in the network.
While a simple metric, such as the fraction of vertices with
at least one globally or locally symmetric counterpart,
would begin to shed a light on this phenomenon,
more fine-grained symmetry metrics are desired to
capture it in more detail.

\bibliographystyle{amsplain}
\bibliography{local_sym}

\end{document}